\documentclass[12pt]{amsart}
\usepackage{graphicx}

\usepackage{tikz-cd}

\usepackage{pgf}

 \usepackage{url}	%url adresses
 \allowdisplaybreaks %align may cross pages

\usepackage{tikz-cd}
\usepackage{pgf}

\newtheorem{teo}{Theorem}[section]
\newtheorem{theorem}[teo]{Theorem}

\newtheorem{corollary}[teo]{Corollary}

\newtheorem{lemma}[teo]{Lemma}

\newtheorem{proposition}[teo]{Proposition}

\theoremstyle{definition}
\newtheorem{definition}[teo]{Definition}
\newtheorem{example}[teo]{Example}

\theoremstyle{remark}
\newtheorem{remark}{Remark}

\numberwithin{figure}{section}%figure numbering

\newcommand{\cat}{\mathrm{cat}}
\newcommand{\Crit}{\mathop{\mathrm{Crit}}}
\newcommand{\D}{\mathrm{D}}
\newcommand{\gcat}{\mathrm{gcat}}
\newcommand{\Ho}{\mathrm{H}}
\newcommand{\hw}{\mathrm{hw}}
\newcommand{\id}{\mathrm{id}}
\newcommand{\im}{\mathrm{im\,}}
\newcommand{\lcp}{\mathrm{l.c.p.}\,}
\newcommand{\R}{\mathbb{R}}
\newcommand{\scat}{\mathrm{scat}}
\newcommand{\SD}{\mathrm{SD}}
\newcommand{\secat}{\mathrm{secat}}
\newcommand{\Sp}{\mathrm{Sp}}
\newcommand{\Sq}{\mathrm{Sq}}
\newcommand{\TC}{\mathrm{TC}}
\newcommand{\Tr}{\mathop{\mathrm{Tr}}}
\newcommand{\Z}{\mathbb{Z}}
\begin{document}
\title{Homotopic distance between maps}
\thanks{The first author was partially supported by MINECO research project MTM2016-78647-P and by Xunta de Galicia ED431C 2019/10 with FEDER funds. The second author was partly supported by Ministerio de Ciencia, Innovaci\'on y Universidades,  grant FPU17/03443.}
\author{E. Mac\'{\i}as-Virg\'os \and D. Mosquera-Lois}
\address{Institute of Mathematics, University of Santiago de Compostela, Spain }
\email{quique.macias@usc.es, david.mosquera.lois@usc.es}

\begin{abstract}We show that both Lusternik-Schnirelmann category and topological complexity are particular cases of a more general notion, that we call {\em homotopic distance}  between two maps. As a consequence,  several properties of those invariants can be proved in a unified way and new results arise. 
%Some generalizations  are envisaged.
\end{abstract}


\maketitle

\section{Introduction}\label{INTRO}
In this paper we prove that well known homotopic invariants like the Lusternik-Schnirelmann category $\cat(X)$ of  the topological space $X$ \cite{CORNEA} or the topological complexity $\TC(X)$ \cite{FARBER} can be seen as particular cases of a more general notion, that we call {\em homotopic distance} between two continuous maps $f,g$, denoted $\D(f,g)$. As a consequence, the proofs of several properties of those invariants can be unified in a systematic way, and new results arise.

It can be conjectured  that this unifying approach will give a new insight about the relationship between  $\cat(X)$ and $\TC(X)$; in particular,   the inequalities we found could serve as new lower bounds for  the difficult problem of computing the category and topological complexity  in explicit examples.

The contents of the paper are as follows:

Section \ref{BASIC} is devoted to the basic definitions and examples. Given two continuous maps $f,g\colon X \to Y$ between topological spaces, we say that $\D(f,g)\leq n$ if there exists an open covering $\{U_0,\dots,U_n\}$ of $X$ such that the restrictions $f_{\vert U_j},g_{\vert U_j}\colon U_j \to Y$ are homotopic maps, for all $j=0,\dots,n$.  Then, by definition, $\cat(X)$ is the distance between $\id_X$ and a constant map.
We show that $\cat(X)$ also equals the homotopic distance $\D(i_1,i_2)$ between the two axis inclusions $i_1,i_2\colon X \to X\times X$ (Proposition \ref{INCLCAT}), while $\TC(X)$ equals the homotopic distance $\D(p_1,p_2)$ between the projections $p_1,p_2\colon X\times X \to X$ (Proposition \ref{PROJECT}).

In Section \ref{PROPERTIES} we prove several properties of the homotopic distance, namely its behaviour under compositions and products, and its homotopical invariance. They imply as particular cases well known inequalities like $\cat(X)\leq \TC(X)\leq \cat(X\times X)$ (Corollary \ref{TCCAT1} and Corollary \ref{TCCAT2}),
or 
$\TC(X\times X^\prime) \leq \TC(X)+\TC(X^\prime)$ (Example \ref{PRODTC}).

In Section \ref{HSP} we study $H$-spaces. For instance we prove that for any pair of maps $f,g \colon G\times G \to G$, where $G$ is an $H$-space,  we have $\D(f,g)\leq \cat (G)$ (Theorem \ref{DISTANCEHSPACE}), thus generalizing Farber-Lupton-Scherer's theorem $\TC(G)=\cat(G)$ \cite{FARBER2,LS}.

In Section \ref{COHOMOLOGY} we give a lower cohomological bound for the homotopic distance, in terms of the length of the cup product, namely $D(f,g)\geq \lcp \mathcal{J}(f,g)$ (Theorem \ref{LCP}), where we denote by $\mathcal{J}(f,g)$ the image of $f^*-g^*$ in $H(X)$. Similar results are well known for the particular cases of $\cat(X)$ or $\TC(X)$.

A better result is obtained after defining the so-called {\em homotopy weight} $\hw_{f,g}(u)\geq 1$ of the non-zero cohomology class $u\in \mathcal{J}(f,g)\subset H(X)$. This generalizes ideas from Fadell-Husseini and other authors, and we are able to prove (Theorem \ref{WEIGHT})  that 
if $u_0\smile\cdots\smile u_k\neq 0$, then $$\D(f,g)\geq \sum_{j=0}^k \hw(u_j).$$  

Section \ref{FIBRATIONS} is about fibrations. We generalize both Varadarajan's result \cite{VARADARAJAN} about the relationship between the LS-category of the total space $E$, the fiber $F$ and the base $B$, and a similar result for the topological complexity, due to Farber and Grant \cite[Lemma 7]{FARBER-GRANT}. Explicitly, we prove (Theorem \ref{FIBRATION}) that 
$$\D(f,g)+1\leq \big(\D(f_0,g_0)+1\big)\big(\cat(B)+1\big)$$ for  fibre preserving maps $f,g$  that induce maps $f_0,g_0\colon F_0 \to F_0^\prime$ between the fibers.

In Section \ref{MAINEXAMPLE}  we show an example of a Lie group $G$ and  two maps $f,g\colon G \to G$ such that $\D(f,g)=2=\lcp H(G)$, while $\TC(G)=\cat(G)=3$.

Finally, Section \ref{FURTHER} contains an overview of possible generalizations, like a version in the simplicial setting or an analog of higher topological complexity.

\subsection*{Acknowledgements}
We are grateful to D.~Tanr\'e for useful comments; to J.~Oprea for pointing out to us several references and suggesting the proof of Th. \ref{DISTPRODUCT}; to J.G. Ca\-rras\-quel-Vera for several remarks, including Th.~\ref{DISTSECAT} and Prop.~\ref{TRIANG}; and to J.A.~Barmak for Example~\ref{TRIANG_FAILS}.

\section{Basic notions}\label{BASIC}
\subsection{Homotopic distance}
Let $f,g\colon X \to Y$ be two continuous maps.

\begin{definition}The {\em homotopic distance} $\D(f,g)$ between $f$ and $g$  is the least integer $n\geq 0$ such that there exists an open covering $\{U_0,\dots,U_n\}$ of $X$ with the property that  $f_{\vert U_j}\simeq g_{\vert U_j}$,   for all $j=0,\dots,n$.  If there is no such covering, we define $\D(f,g)=\infty$.
\end{definition}

Notice that:  
\begin{enumerate}
	\item
	$\D(f,g)=\D(g,f)$.
	\item
	$\D(f,g)=0$ if and only if the maps $f,g$ are homotopic.
\end{enumerate}

In fact, the homotopic distance only depends on the homotopy class.

\begin{proposition}\label{HOMOT}If $f\simeq f^\prime$ and $g\simeq g^\prime$ then $\D(f,g)=\D(f^\prime,g^\prime)$.
\end{proposition}

Later (Proposition \ref{INVARIANCEPAIR}) we shall show that $\D(-,-)$ is an invariant of the homotopy class a pair of maps.

Sometimes, the following result is useful.

\begin{proposition}[Sub-additivity property]\label{lema:subadditivity}
	Given two continuous maps $f,g\colon X\to Y$ and a finite open covering $\{U_0,\dots,U_n\}$ of $X$, we have $$\D(f,g)\leq \sum_{k=0}^n\D(f_{\vert U_k},g_{\vert U_k})+ n.$$
\end{proposition}

\begin{example}Let $X=S^1$ be the Lie group of unit complex numbers. The distance between the identity $z$ and the inversion $1/z$ is $1$.  Let $X= S^2$ be the unit sphere. The distance between the identity and the antipodal map is $1$ (see Corollary \ref{CATDOM}).\end{example}

\begin{example}Let $G$ be the unitary group $U(2)$. The distance between the identity $\id_G$ and the inversion $I(A)=A^*$ is $\D(\id_G,I)=2$ (see Equation \eqref{U2}).
\end{example}

The two key examples of homotopic distance are Lusternik-Schni\-rel\-mann category  \cite{CORNEA}  and Farber's topological complexity \cite{FARBER}, as we shall show in the next paragraphs. 
\subsection{Lusternik-Schnirelmann category}
Assume the space $X$ to be path-connected. An open set $U\subset X$ is {\em categorical in $X$} if the inclusion is null-homotopic. The (normalized) LS-category $\cat(X)$ is the least integer $n\geq 0$ such that $X$ can be covered by $n+1$ categorical open sets. Then, $\cat(X)$ is the homotopic distance between the identity $\id_X$   and any constant map, that is, $\cat(X)=\D(\id_X,*)$.

More generally, the Lusternik-Schnirelmann category of the map $f\colon X \to Y$ \cite[Exercise 1.16, p.~43]{CORNEA}  is the distance between $f$  and any constant map, $\cat(f)=\D(f,*)$, when $Y$ is path-connected. For instance, the category of the diagonal $\Delta_X\colon X \to X \times X$ equals $\cat(X)$. 

Given a base point $x_0\in X$ we define the inclusion maps $i_1,i_2\colon X\to X\times X$ as $i_1(x)=(x,x_0)$ and $i_2(x)=(x_0,x)$. 

\begin{proposition}\label{INCLCAT}
	The homotopic distance between $i_1$ and $i_2$ equals the LS-category of $X$, that is, $\D(i_1,i_2)=\cat(X)$.
\end{proposition}

\begin{proof}
	First, we show that $\D(i_1,i_2)\leq \cat(X)$. If an open subset $U\subset X$ is categorical, let  $\mathcal{F}\colon U\times [0,1]\to X$ be the homotopy between the inclusion and the constant map to $x_0\in X$, i.e. $\mathcal{F}(x,0)=x$ and $\mathcal{F}(x,1)=x_0$. We define a homotopy $\mathcal{H}\colon U\times [0,1] \to X\times X$ between $(i_1)_{\vert U}$ and $(i_2)_{\vert U}$ as
	$$\mathcal{H}(x,t)=
	\left\{	
	\begin{array}{lr}
	(\mathcal{F}(x,2t),x_0) & \mathrm{if\ }\  0\leq t\leq 1/2,\\
	(x_0,\mathcal{F}(x,2-2t))  & \mathrm{if\ }\ 1/2\leq t \leq 1.
	\end{array}
	\right.$$
	%This map is continuous because
	%$$\mathcal{H}(x,1/2)=\big(\mathcal{F}(x,1),x_0\big)=(x_0,x_0)=\big(x_0,\mathcal{F}(x,1)\big)=(x_0,x_0).$$
	%Moreover,
	%$$\mathcal{H}(x,0)=\big(\mathcal{F}(x,0),x_0\big)=(x,x_0)=i_1(x)$$
	%while
	%	$$\mathcal{H}(x,1)=\big(x_0,\mathcal{F}(x,0)\big)=(x_0,x)=i_2(x).$$
	Second, we show that $\cat(X)\leq \D(i_1,i_2)$. Assume that there is a homotopy $\mathcal{H}\colon U\times [0,1]\to X \times X$ between $(i_1)_{\vert U}$ and $(i_2)_{\vert U}$, i.e. $\mathcal{H}(x,0)=(x,x_0)$ and $\mathcal{H}(x,1)=(x_0,x)$. Let  $p_1\circ \mathcal{F}$ be the first component of $\mathcal{F}$. Then $p_1\circ \mathcal{F}$  is a homotopy between the inclusion $U \subset X$ and the constant map $x_0$.  
\end{proof}
\subsection{Topological complexity}
Let $\pi\colon PX \to X\times X$ be the path fibration sending each continuous path $\gamma\colon [0,1]\to X$ on $X$ to its initial and final points, $\pi(\gamma)=\big(\gamma(0),\gamma(1)\big)$. By definition, the (normalized) topological complexity $\TC(X)$ of $X$ is the least integer $n$ such that $X\times X$ can be covered by $n+1$ open subsets $U_j$ where the fibration $\pi$  admits a continuous local section.

\begin{proposition}\label{PROJECT}The topological complexity of $X$ equals the homotopic distance between the two projections $p_1,p_2\colon X\times X \to X$, that is,
	$\TC(X)=\D(p_1,p_2)$.
\end{proposition}

This result will be a consequence of Theorem \ref{DISTSECAT} below.

\subsection{\v{S}varc genus}
Both $\cat(X)$ and $\TC(X)$ are particular cases of the \v{S}varc  genus (also called sectional category) of some fibrations. Explicitly \cite{CORNEA}
the {\em \v{S}varc genus} $\secat(\pi)$ of a fibration $\pi\colon E \to B$  is the minimum integer  $n\geq 0$ such that the base $B$ can be covered by open sets $V_0,\dots,V_n$ with the property that over each $V_j$ there exists a local section $s$ of $\pi$. For instance, $\cat(X)$ is the \v{S}varc genus of the fibration $\pi_0\colon P_0X \to X$ sending each path $\gamma$ with initial point $x_0$ into the end point $\gamma(1)$.

What follows is an interpetation of the homotopic distance in terms of the \v{S}varc genus.

\begin{theorem}\label{DISTSECAT} Let $f,g\colon X \to Y$ be two maps, and consider the pull-back $q\colon P\to X$ of the path fibration $\pi\colon PY \to Y\times Y$ by the map $(f,g)\colon X \to Y\times Y$: 
	$$\begin{tikzcd}
	P \arrow[d,"q"]\arrow[r] & PY\arrow[d,"{\pi}"]\\
	X \arrow[r,"{(f,g)}"] & {Y\times Y}
	\end{tikzcd}$$
	Then $D(f,g)=\secat(q)$. 
\end{theorem}

\begin{proof}The elements of $P$ are the pairs $(x,\gamma)$ where $x\in X$ and $\gamma$ is a path on $Y$ with $\gamma(0)=f(x)$ and $\gamma(1)=g(x)$. The map $q$ is the projection onto the first factor. Then, if $U\subset X$ is an open set where there exists a homotopy $\mathcal{H}\colon U\times I \to Y$ between $f_{\vert U}$ and $g_{\vert U}$, we can define a section $s\colon X \to P$ as 
	$s(x)=\big(x,\mathcal{H}(x,-)\big)$.
	Then $\secat(q)\leq D(f,g)$.
	
	Conversely, if there is a map $s\colon U \to P$ such that $q\circ s$ is the inclusion $i_U\colon U \subset X$, we have
	$s(x)=(x,\gamma_x)$ for some path $\gamma_x$ from $f(x)$ to $g(x)$. Then,   the homotopy 
	$\mathcal{H}(x,t)=\gamma_x(t)$ proves that $f\simeq g$ on $U$.
	%That all the maps involved are continuous follows from standard arguments.
\end{proof}

As a consequence,  if we take $f=p_1$ and $g=p_2$ to be the projections from $X\times X\to X$ we have $(f,g)=\id_{X\times X}$ and $q=\pi$, thus proving Proposition \ref{PROJECT}, that is,
$$D(p_1,p_2)=\secat(\pi)=\TC(X).$$

\section{Properties}\label{PROPERTIES}
\subsection{Compositions} We now prove several elementary properties, starting with the behaviour of the homotopic distance under compositions. Several known properties of $\cat$ and $\TC$ can be deduced from our general results.

\begin{proposition}\label{IZQ}Suppose we have maps $f,g\colon X \to Y$ and $h\colon Y\to Z$. Then 
	$$\D(h\circ f,h\circ g)\leq \D(f,g).$$
\end{proposition}

\begin{proof} Let $\D(f,g)\leq n$ and let $X=U_0\cup\dots\cup U_n$ be an open covering with $f_j=f_{\vert U_j}$ homotopic to $g_j=g_{\vert U_j}$. Then $\D(hf,h g)\leq n$ beause
	$$(h\circ f)_j=h\circ f_j\simeq h\circ g_j=(h\circ g)_j.$$ 
\end{proof}

\begin{corollary}\label{cor_d_upper_bound_domain}
	Let $f\colon X \to Y$ be a map with path-connected domain $X$. Then $\cat(f)\leq \cat(X)$. 
\end{corollary}

\begin {proof}Take $\id_X$ and a constant map $x_0$. Then
$\D(f\circ\id_X,f(x_0))\leq \D(\id_X,x_0)$.
\end{proof}

\begin{proposition}\label{DER}Suppose we have maps $f,g\colon X \to Y$ and $h\colon Z\to X$.  Then 
$$\D(f\circ h,g\circ h)\leq \D(f,g).$$
\end{proposition}

\begin{proof} Let $\D(f,g)\leq n$ and let $X=U_0\cup\dots\cup U_n$ be an open covering with $f_j\simeq g_j\colon U_j \to Y$. Let $V_j=h^{-1}(U_j)\subset Z$. The restriction $h_j\colon V_j \to X$ can be written as the composition of a map  $\bar h_j\colon V_j\to U_j$, where $\bar h_j(x)=h(x)$, and the inclusion $I_j\colon U_j\subset X$. Then we have that 
$$(f\circ h)_j=f_j\circ \bar h_j \simeq g_j\circ \bar h_j=g\circ I_j\circ \bar h_j=g\circ h_j=(g\circ h)_j,$$
hence $D(fh,gh)\leq n$.
\end{proof}

\begin{corollary}If $f\colon X \to Y$ is a continuous map with a path-connected codomain $Y$, then $\cat(f)\leq \cat(Y)$.
\end{corollary}

\begin{proof}Take $\id_Y$ and a constant map $y_0$.
Then
$D(\id_Y\circ f, y_0\circ f)\leq D(\id_Y,y_0)$.
\end{proof}

The latter result result can be extended. 

\begin{corollary}\label{cor:d_upper_bound_codomain}
Let $f,g\colon X \to Y$ be continuous maps with a path-connected codomain $Y$. Then 
$$\D(f,g)+1\leq (\cat (f)+1)(\cat(g)+1).$$
\end{corollary}

\begin{proof}
Denote by $y_0$ a constant map from $X$ to $Y$.  Assume that $\cat(f)=\D(f,y_0)\leq m$, $\cat(g)=\D(g,y_0)\leq n$ and let $\{U_i\}_{i=0}^m$, $\{V_j\}_{j=0}^n$ be the corresponding coverings of $X$. The open sets $W_{i,j}=U_i\cap V_j$  cover $X$.   Moreover, $f \simeq y_0\simeq g$ on $W_{i,j}$, so $D(f,g)\leq m\times n$. 
\end{proof}

The latter result will be greatly improved for normal spaces (see the remark after Proposition \ref{TRIANG}).

\begin{corollary}\label{EJY0}Let $f,g\colon X \to Y$. Let $y_0\in Y$ be a point, $X_0=g^{-1}(y_0)$ the fiber and $f_0=f_{\vert X_0}$ the restriction. Then $ \D(f_0,y_0)\leq \D(f,g)$.\end{corollary}

\begin{corollary}[\cite{FARBER}\label{TCCAT1}] $\cat(X)\leq \TC(X)$.
\end{corollary}   

\begin{proof}In Corollary~\ref{EJY0} take $f=p_1,g=p_2\colon X\times X \to X$ and a point $x_0\in X$. Then $(p_1)_0=\id_X$ and $$\cat(X)=\D(\id_X, x_0)\leq \D(p_1,p_2)=\TC(X).$$

Another proof follows directly from Proposition \ref{DER} by considering the inclusion maps $i_1,i_2\colon X\to X\times X$, so $$\D(*,\id_X)=\D(p_1\circ i_2,p_2\circ i_2)\leq \D(p_1,p_2).$$
\end{proof}

In the next Proposition we shall prove a non-obvious inequality. 

\begin{proposition}\label{IGU}Let $h,h^\prime \colon Z \to X$ and $f,g\colon X\to Y$  be maps such that $f\circ h^\prime \simeq g\circ h^\prime$. Then
$$ \D(f\circ h, g\circ h)\leq D(h,h^\prime).$$
\end{proposition}

\begin{proof}
We denote the homotopy $f\circ h^\prime \simeq  g\circ h^\prime$ by $\mathcal{G}\colon Z\times [0,1] \to Y$. Assume $\D(h,h^\prime)\leq n$, and let $Z=U_0\cup\cdots\cup U_n$ be a covering such that, for all $j$,  $h_{\vert U_j}\simeq h^\prime_{\vert U_j}$ by a homotopy $\mathcal{H}\colon U_j\times [0,1]\to X$. Let us define the map $\mathcal{F}\colon U_j\times [0,1] \to Y$ as follows:
$$%\label{FORMULA}
\mathcal{F}(x,t)= 
\left\{
\begin{array}{ll}   
f\mathcal{H}(x,3t), & \mathrm{if\ }\  0 \leq t \leq 1/3,  \\
\mathcal{G}(x,3t-1), & \mathrm{if\ }\  1/3 \leq  t \leq 2/3,  \\
g\mathcal{H}(x,3-3t), & \mathrm{if\ }\  2/3 \leq  t \leq 1.\\
\end{array}
\right.
$$

%This map is continuous because
%$$\mathcal{F}(x,1/3)=\mathcal{H}(x,1)=fh^\prime(x)=\mathcal{G}(x,0)$$
%and 
%$$\mathcal{F}(x,2/3)=\mathcal{G}(x,1)=g h^\prime(x)=g\mathcal{H}(x,1).$$
%Moreover,
%$$\mathcal{F}(x,0)=f\mathcal{H}(x,0)=fh(x)$$
%while
%$$\mathcal{F}(x,1)=g\mathcal{H}(x,0)=g h(x)$$
%thus showing that
%$fh\simeq_\mathcal{F} g h$ on $U_j$. Hence, $\D(fh,g h)\leq n$.
This map shows that $fh\simeq g h$ on $U_j$. Hence, $\D(fh,g h)\leq n$.
\end{proof}

\subsection{Domain and codomain}
Recall that   the {\em geometric LS-category} of   $X$,  denoted by $\gcat(X)$, is the least integer $n\geq 0$ such that $X$ can be covered by $n+1$  open sets which are {\em contractible in themselves}.
This subtle difference with the LS-category ---where the open sets are contractible in the ambient space--- is important,  because in general $\gcat$ is not a homotopy invariant.
Since any map with a contractible domain is homotopic to a constant map, it is obvious that   $\D(f,g)\leq \gcat(X)$ for any pair of  continuous maps $f,g\colon X\to Y$.	

The inequality $\D(f,g)\leq \cat(X)$ is much less evident. 

\begin{corollary}\label{CATDOM}Let $f,g \colon X \to Y$ be two maps with  path-connected domain $X$ and codomain $Y$. Then
$$ \D( f, g)\leq \cat(X).$$
\end{corollary}
\begin{proof} In Theorem \ref{DISTSECAT}, we have $\D(f,g)=\secat(q)$. Since $q$ is a fibration, the homotopy lifting property implies that $\secat(q)\leq \cat(X)$.

Another  proof of this Corollary follows from Proposition \ref{IGU}: if $Z=X$, $h=\id_X$ and $h^\prime=x_0$ a constant map, then the constant maps $f(x_0),g(x_0)\colon X \to Y$ are homotopic because $Y$ is path-connected, so 
$$\D(f,g)=\D(f\circ \id_X,g\circ\id_X)\leq D(\id_X,x_0)=\cat(X).$$
\end{proof}

\begin{corollary}[\cite{FARBER}]\label{TCCAT2} $\TC(X)\leq \cat(X\times X)$.
\end{corollary}

\begin{proof}In Corollary \ref{CATDOM} take the maps $p_1,p_2\colon X\times X \to X$. 
%Then $\TC(X)=\D(p_1,p_2)\leq \cat(X\times X)$.
\end{proof}

For the codomain, we have the following result.

\begin{proposition}\label{CODOMAIN}For maps $f,g\colon X \to Y$ we have $\D(f,g)\leq \TC(Y)$.
\end{proposition}

\begin{proof}This follows from Theorem \ref{DISTSECAT}, because if the fibration $q$ is a pullback of the fibration $\pi$ then  $\secat(q)\leq \secat(\pi)$, which is exactly $\D(f,g)\leq \TC(Y)$.\end{proof}

Notice that in general it is not true that $\D(f,g)\leq \cat(Y)$. In fact, by taking the projections $p_1,p_2\colon Y\times Y \to Y$ this would imply that $\TC(Y)\leq\cat(Y)$, which is not true in general. However, this is true for $H$-spaces (see Section \ref{HSP}).

\subsection{Invariance}
We  prove the homotopy invariance of the homotopic distance.

\begin{proposition}\label{EQUIV1}Let $f,g\colon X\to Y$ be maps and let $\alpha\colon Y \to Y^\prime$ be a map with a left homotopy inverse. Then $\D(\alpha\circ f,\alpha\circ g)=\D(f,g)$.
\end{proposition}
\begin{proof}
By Propositions \ref{IZQ} and   \ref{HOMOT}, we have
$$\D(f,g)\geq \D(\alpha f,\alpha  g)\geq \D(\beta \alpha  f,\beta \alpha  g)=\D(f,g),$$
because $\beta \alpha \simeq \id_Y$ implies $\beta\alpha f \simeq f$ and $\beta\alpha g\simeq g$.
\end{proof} 
Analogously
\begin{proposition}\label{EQUIV2}Let $f,g\colon X\to Y$ be maps and let $\beta\colon X^\prime \to X$ be a map with a right homotopy inverse. Then $\D(f\circ \beta,g\circ \beta)=\D(f,g)$.
\end{proposition}

As a consequence, $\D(\ ,\ )$ is a homotopy invariant in the following sense 

\begin{proposition}\label{INVARIANCEPAIR}Assume that there exist homotopy equivalences $\beta \colon X^\prime\simeq X$ and $\alpha\colon Y\simeq Y^\prime$ such that $f\colon X \to Y$ (resp. $g$) and $f^\prime\colon X^\prime \to Y^\prime$ (resp. $g^\prime)$ verify $\alpha\circ f \circ\beta\simeq f' $ (resp. $\alpha\circ g\circ \beta \simeq g' $):
$$	\begin{tikzcd}
		X \arrow[r,shift left, "f"] \arrow[r, shift right,"g"']   & Y \arrow[d, "\alpha"] \\
		X^\prime  \arrow[u, "\beta"] \arrow[r,shift left, "f^\prime"]\arrow[r,shift right,"g^\prime"'] &   Y^\prime
	\end{tikzcd}$$
Then $\D(f,g)=\D(f^\prime,g^\prime)$. 
\end{proposition}

%\begin{proof}By Prop. \ref{DER} and Prop. \ref{HOMOT} ,
%
%$$\D(f,g)\geq \D(f\beta,g\beta)\geq \D(f\beta\alpha,g\beta\alpha)=D(f,g).\qedhere$$
%\end{proof}

\begin{corollary} Both $\cat(\ )$ and $\TC(\ )$ are homotopy invariant, that is, if there exist homotopy equivalences $X\simeq X^\prime$, then $\cat(X)=\cat(X')$ and $\TC(X)=\TC(X')$.
\end{corollary}

\subsection{Normal spaces}
For normal spaces we shall use the following strikingly general Lemma, proved by Oprea and Strom \cite[Lemma 4.3]{JOHNJEFF}. 
\begin{lemma}
\label{JOHNJEFF}
Let $Z$ be a normal space with two open covers $\,\mathcal{U} = \{U_0, \dots, U_m\}$ and $\mathcal{V} = \{V_0, \dots, V_n\}$
such that each set of $\,\mathcal{U}$ satisfies Property (A) and each set of $\mathcal{V}$ satisfies Property
(B). Assume that Properties (A) and (B) are inherited by open subsets and
disjoint unions. Then $Z$ has an open cover
$\mathcal{W} = \{W_0,\dots, W_{m+n}\}$
by open sets, satisfying both Property (A) and Property (B).
\end{lemma}

As a first consequence, we prove that the homotopic distance verifies the triangular inequality, thus being a true distance in the space of homotopy classes.

\begin{proposition}\label{TRIANG}Let $f,g,h\colon X \to Y$ be maps defined on the normal space $X$. Then
$$\D(f,h)\leq \D(f,g)+\D(g,h).$$
\end{proposition}

\begin{proof}
Let $\D(f,g)=m$ and $\D(g,h)=n$. Take coverings $\{U_0,\dots,U_m\}$ and $\{V_0,\dots,V_n\}$ of $X$ such that $f_{\vert U_i}\simeq g_{\vert U_i}$ for all $i=0,\dots m$ and $g_{\vert V_j}\simeq h_{\vert V_j}$ for all $j=0,\dots,n$. Clearly these properties  are closed for open subsets and disjoint unions. Then, by Lemma \ref{JOHNJEFF}, there is a third covering  $\{W_0,\dots, W_{m+n}\}$ such that
$f_{\vert W_k}\simeq g_{\vert W_k}\simeq h_{\vert W_k}$,
for all $k$, thus proving that $\D(f,h)\leq m+n$.
\end{proof}

\begin{remark} Proposition \ref{TRIANG} does not hold in general for arbitrary topological spaces, as we shall show in subsection \ref{FINITESPACES}.
\end{remark}

Note that Corollary \ref{cor:d_upper_bound_codomain} could be improved (in normal spaces), because
$\D(f,g)\leq \D(f,*)+\D(*,g)$
means that
$\D(f,g)\leq \cat(f)+\cat(g)$.

\medskip

Another  result also follows from Lemma \ref{JOHNJEFF}.
\begin{proposition}Let $X$ be a normal space. For maps $f,g\colon X \to Y$ and $f^\prime,g^\prime\colon Y \to Z$ we have
$$D(f^\prime\circ f, g^\prime \circ g)\leq D(f,g)+D(f^\prime, g^\prime).$$
\end{proposition}

\begin{proof}
If $D(f,g)=m$ there is a covering $U_0,\dots,U_m$ of $X$ where $f\simeq g$.  It follows that $g^\prime\circ f\simeq g^\prime \circ g$ for this covering. Clearly, the criterion of the Lemma is verified.

Now, if $D(f^\prime,g^\prime)=n$, there is a covering $V_0,\dots,V_n$ of $Y$ where $f^\prime \simeq g^\prime$. But then, the covering $f^{-1}(V_0),\dots,f^{-1}(V_n)$ of $X$ verifies $f^\prime\circ f\simeq g^\prime\circ f$. This property also fullfils the criterion.

Hence there is a third covering $W_0,\dots,W_{m+n}$ of $X$ where $f^\prime\circ f \simeq g^\prime\circ f\simeq g^\prime\circ g$, which implies $D(f^\prime\circ f,g^\prime\circ g)\leq m+n$.
\end{proof}

The latter result generalizes Propositions \ref{IZQ} and \ref{DER}, at least for normal spaces, because $D=0$ for homotopic maps.

\subsection{Products}
We study   the behaviour of the homotopic distance under products. 

\begin{theorem}\label{DISTPRODUCT}
Given $f,g\colon X\to Y$ and $f',g'\colon X'\to Y'$, assume that the space  $X\times X'$ is normal. Then the maps $f\times f^\prime, g\times g^\prime\colon X\times X^\prime \to Y\times Y^\prime$ verify that $$\D(f\times f',g\times g')\leq \D(f,g) + \D(f',g').$$
\end{theorem}
It is possible to give a proof identical to that given in \cite[Section 1.5]{CORNEA} for the particular case of LS-category, just by replacing the  notion of categorical sequence by a similar notion of {\em homotopical sequence}.
However, a much simpler proof
follows from  Lemma \ref{JOHNJEFF}.
%\begin{remark}
%The hypothesis that $X\times X'$ %is completely normal can be %replaced by $X\times X'$ being %paracompact \cite[Exercise %1.11]{CORNEA}.
%\end{remark}

\begin{proof}[of Theorem \ref{DISTPRODUCT}]
For an open set $\Omega\subset X\times X^\prime$ consider the following Property (A): there is a homotopy $f\times \id_{X^\prime}\simeq g\times \id_{X^\prime}$ on $\Omega$. Clearly this property is inherited by open subsets and disjoint unions.

Analogously, the open set $\Omega^\prime\subset X\times X^\prime$ will satisfy Property (B)
if there is a homotopy $\id_X\times f^\prime\simeq \id_X\times g^\prime$ on $\Omega^\prime$. 

Now, suppose that $\D(f,g)=m$ and $\D(f',g')=n$, and take open coverings $\{U_0,\dots,U_m\}$ of $X$ and $\{V_0,\dots, V_n\}$ of $X^\prime$ such that $f\simeq g$ on each $U_i$, $i=0,\dots,m$, and $f^\prime\simeq g^\prime$ on each $V_j$, $j=0\dots n$. 

Then $\mathcal{U}=\{U_i\times X^\prime\}$ is a covering of $X\times X^\prime$ verifiying Property (A) and $\mathcal{V}=\{X\times V_j\}$ is a covering verifying Property (B). By Lemma \ref{JOHNJEFF}, there is a third covering $\mathcal{W}=\{W_0,\dots,W_{m+n}\}$ of $X\times X^\prime$ such that each $W=W_k$, $k=0\dots,m+n$, verifies both properties, namely,
$f\times \id_{X^\prime}\simeq g\times \id_{X^\prime}$ on $W$ (by a certain homotopy $\mathcal{H}\colon \colon W\times I \to Y\times Y^\prime$), and $\id_X\times f^\prime\simeq \id_X\times g^\prime$ (by a certain homotopy $\mathcal{H}^\prime$). Consider the homotopy
$$\mathcal{H}^{\prime\prime}=(p_Y\circ \mathcal{H},p_{Y^\prime}\circ \mathcal{H^\prime})\colon W\times I \to Y \times Y^\prime.$$
We have 
$$
\mathcal{H}^{\prime\prime}_0(z)=\big(p_Y\mathcal{H}_0(z),p_{Y^\prime} \mathcal{H}^\prime_0(z)\big)=\big(fp_X(z),f^\prime p_{X^\prime}(z)\big)
=(f\times f^\prime)(z).
$$
Analogously
$\mathcal{H}^{\prime\prime}_1 =(g\times g^\prime)$,
thus showing that $f\times f^\prime \simeq g\times g^\prime$ on $W$.

We have then proved  that $\D(f\times f^\prime \simeq g\times g^\prime)\leq m+n$, as stated.
\end{proof}

\begin{example}\label{PRODLS}
Set $f\colon X\to X$ and $f'\colon X'\to X'$ to be the identity maps and $g\colon X\to X$ and $g'\colon X'\to Y'$ to be constant maps. Then $$\cat(X\times X')\leq \cat(X) + \cat(X').$$
\end{example}

\begin{example}\label{PRODTC}
Set $f\colon X \times X\to X$ and $f'\colon X' \times X' \to X'$ to be the projection maps onto the first factor and $g\colon X \times X\to X$ and $g'\colon X' \times X' \to X'$ to be the projection maps onto the second factor. Then $$\TC(X\times X')\leq \TC(X) + \TC(X').$$
\end{example}

\subsection{Finite topological spaces}\label{FINITESPACES}
It has been shown (Prop. \ref{TRIANG}) that homotopic distance satisfies the triangular inequality under the assumption that the domains of the  maps involved are normal spaces. This subsection is devoted to showing this does not hold  for non-discrete finite topological spaces, despite they are endowed with a rich combinatorial structure which in some contexts offsets the lack of separation conditions.  

We  recall some basic facts about finite topological spaces; for a detailed exposition we refer the reader to \cite{BARMAK}. Finite posets and finite $T_0$-spaces are in bijective correspondence.  If $(X, \leq)$ is a poset, a basis for a topology on $X$ is given by the sets $$U_x=\{y\in X: y\leq x\}, \quad x\in X.$$   Conversely, if $X$ is a finite $T_0$-space, define, for each $x\in X$, the {\it minimal open set} $U_x$ as the intersection of all open sets containing $x$.  Then $X$ may be given a poset structure by defining $y\leq x$ if and only if $U_y\subset U_x$. Given two finite spaces $X$ and $Y$, the product topology is given  by the basic open sets 
$$U_{(x,y)}=U_x\times U_y, \quad (x,y)\in X\times Y.$$

\begin{example} \label{TRIANG_FAILS}
Let $S$ be the finite space corresponding to the poset depicted in Figure \ref{fig_S1}. 

\begin{figure}[htbp]
	\centering
	\includegraphics[scale=0.6]{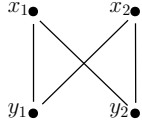}
	\caption{The finite topological space $S$.}
	\label{fig_S1}
\end{figure}

Consider the finite space $X=S\times S$ and the continuous maps $f,g,h\colon X\to X$ given by $f=\id_X$, $g=\id_S\times c$ and $h=c\times c$ where $c\colon S\to S$ is a constant map. Recall from \cite{BARMAK} that for any finite space $Y$ and $y\in Y$, the subspace $U_y$ is contractible. Therefore  $\{S\times U_{x_1}, S\times U_{x_2}\}$ is an open cover of $X$ such that the restrictions of $f$ and $g$ to each of the members of the cover are homotopic. This proves that $\D(f,g)\leq 1$. A symmetrical argument shows that $\D(g,h)\leq 1$. However $\D(f,h)=\cat(X)\geq 3$ \cite[Example 3.5]{TANAKA}.  Therefore, the maps $f,g$ and $h$ do not satisfy the triangular inequality. 
\end{example}

\section{$H$-spaces}\label{HSP}
A well known result from Farber \cite[Lemma 8.2]{FARBER2} states that for a Lie group $G$ the topological complexity $\TC(G)$ equals the LS-category $\cat(G)$. This result 
was later extended to all $H$-spaces by Lupton and Scherer \cite{LS}.

Here, an $H$-space is a topological space $G$ endowed with a {\em multiplication} $\mu\colon G\times G \to G$, a {\em division } $\delta \colon G \times G \to G$ and an identity element $x_0\in G$ such that 
$\mu(p_1,\delta)\simeq p_2$ and $\mu(-,x_0)\simeq \id_G$. Note that we do not ask  the multiplication to be associative.

This definition   is inspired by the discussion in \cite[proof of Theorem 1] {LS} of the results of James \cite{JAMES}. As an example, 
let $G$ be a Lie group, with multiplication $\mu(x,y)=xy$ and division $\delta(x,y)=x^{-1}y$.

Farber and Lupton-Scherer  results are particular cases of the following theorem.

\begin{theorem}\label{DISTANCEHSPACE}Let $G$ be a path-connected $H$-space and let $f,g \colon G\times G \to G$ be two maps. Then $\D(f,g)\leq \cat(G).$\end{theorem}

In fact we know that $\D(f,g)\leq \TC(G)$ (Proposition \ref{CODOMAIN}), so the latter Theorem is equivalent to Lupton-Scherer result. For the sake of completeness we shall give a direct proof.

\begin{proof}Let $U\subset G$ be a categorical open set, that is, $i_U\simeq x_0$, and consider the preimage $\Omega\subset G \times G$ of $U$ by the map $\delta\circ (f,g)\colon G \times G \to G$. Then 
	$$p_2\circ (f,g)\circ i_\Omega \simeq \mu\circ (p_1,\delta)\circ (f,g)\circ i_\Omega,$$
	that is,
	$$g_{\vert \Omega}\simeq \mu\circ (f_{\vert \Omega},\delta\circ (f,g)\circ i_\Omega).$$
	But $\delta\circ (f,g)\circ i_\Omega$ factors through $i_U$, by the definition of $\Omega$, so it is homotopic to the constant map $(x_0)_{\vert \Omega}\colon \Omega \to G$. Then
	$$g_{\vert \Omega}\simeq \mu\circ (f_{\vert \Omega},x_0) \simeq  f_{\vert \Omega}.$$
\end{proof}

\begin{corollary}[\cite{LS}] \label{TCCAT}
	For a path-connected $H$-space $G$ we have $\TC(G)=\cat(G)$.\end{corollary}
\begin{proof}Take $f=p_1$ and $g=p_2$ and apply Theorem \ref{PROJECT}. Then $\TC(G)\leq \cat(G)$. The other inequality was proven in Corollary \ref{TCCAT1}.\end{proof}

In fact, we have 
\begin{proposition}\label{MULTDIV}In a Lie group, the distance between the multiplication and the division equals the distance between the identity $\id_G$ and the inversion map $I\colon G \to G$, $I(x)=x^{-1}$, that is,
	$\D(\mu,\delta)=\D(\id_G,I)$.
\end{proposition}

\begin{proof}Let $x_0=e$ be the identity, and consider the map $i_1(x)=(x,x_0)$. Then $\mu\circ i_1=\id_G$ and $\delta\circ i_1=I$. From Proposition \ref{DER} it follows that
	$\D(\id_G, I)\leq \D(\mu,\delta).$
	
	On the other hand, we have $\mu=\mu\circ (\id_G,\id_G)$ and $\delta=\mu\circ(I,\id_G)$. If $U\subset G$ is an open set where $i_U\simeq I_{\vert U}$, then $\mu_{\vert U\times G}\simeq \delta_{\vert U\times G}$, so $U\times G$ is a homotopy domain for $\mu$ and $\delta$. Since $G=\bigcup_{i=0}^n U_j$ implies $G\times G =\bigcup_{i=0}^n (U_j\times G)$, it follows that $\D(\mu,\delta)\leq \D(\id_G,I)$.\end{proof}

Note that Corollary \ref{TCCAT} and Proposition \ref{CODOMAIN} imply that $D(\mu,\delta)\leq \cat(G)$.

\begin{proposition}
	In any Lie group, the distance $\D(\mu_a,\mu_b)$ between two power maps $\mu_a,\mu_b\colon G \to G$, where $\mu_c(x)=x^c$, equals $\D(\mu_{a-b},e)$. 
\end{proposition}

\begin{proof}If $\mu_a\simeq\mu_b$ on the open set $U\subset G$, then
	$\mu_{a-b}=\mu\circ (\mu_a,\mu_{-b}) \simeq \mu\circ (\mu_b,\mu_{-b})=e$ on $U$,
	where $\mu\colon G \times G \to G$ is the multiplication. Then $D(\mu_{a-b},e)\leq D(\mu_a,\mu_b)$.
	
	On the other hand, if $\mu_{a-b}\simeq e$ on the open subset $U\subset G$, then
	$\mu_a =\mu\circ (\mu_{a-b},\mu_b)\simeq \mu\circ (e,\mu_b)=\mu_b$ on $U$, hence proving the reverse inequality.
\end{proof}

\section{Cohomology}\label{COHOMOLOGY}
\subsection{Cup length}
For the LS category it is well-known \cite{CORNEA} that 
$$\lcp H(X;R)\leq \cat (X),$$ where $\lcp$ denotes the length of the cup product of the cohomology (with coefficients in any commutative ring $R$ with unit). 

Analogously, Farber \cite{FARBER} proved that $\lcp \ker\Delta^*\leq \TC(X)$.
When the coefficients are a field $K$,  
$\ker\Delta^*$ is isomorphic to the kernel of the cup product
$$ H(X;K)\otimes H(X;K) \stackrel{\smile}{\longrightarrow} H(X;K).$$
We shall give a general cohomological lower bound for the homotopic distance between two maps.

Let $f,g\colon X \to Y$ be two maps and let $f^*,g^*\colon H(Y;R)\to H(X;R)$ be the induced morphisms in cohomology (for an arbitrary unitary commutative ring of coefficients).
We denote by $\mathcal{J}(f,g)\subset H(X;R)$ the image of the linear morphism $f^*-g^*\colon H(Y;R)\to H(X;R)$.

\begin{definition}
	We denote by   $\lcp \mathcal{J}(f,g)$  the least integer $k$ such that any product $u_0\smile\cdots \smile u_k$ of elements of $\mathcal{J}(f,g)$ is null in $H(X)$. 
\end{definition}

Note that we do not ask $\mathcal{J}(f,g)$ to be a ring. Also note that 
$$\lcp \mathcal{J}(f,g)\leq \lcp H(X;R).$$

\begin{theorem}\label{LCP}Let $\mathcal{J}(f,g)\subset H(X;R)$ be  the image of the morphism $f^*-g^*\colon H(Y;R)\to H(X;R)$. Then $\lcp \mathcal{J}(f,g)\leq \D(f,g)$.
\end{theorem}

\begin{proof}Assume $\D(f,g)\leq n$ and let $X\times X =U_0\cup\dots\cup U_n$ be a covering  such that the restrictions of $p_1$ and $p_2$ to each open set $U_k$, $k=0,\dots,n$, are homotopic. For $U=U_k$ let us consider the long exact sequence of the pair  $(X, U)$ (from now on we shall not make explicit the ring $R$):
	$$\begin{tikzcd}
	& & {H}^{m}(Y) \arrow{d}[swap]{f^*-g^*} \arrow[dr, "{(f_{\vert U})^*-(g_{\vert U})^*}=0"]\\
	\cdots \arrow[r] & {H}^{m}(X,U) \arrow[r,"{j_U}"] &  {H}^{m}(X)  \arrow[r,"(i_U)^*"] & {H}^{m}(U) \arrow[r] & \cdots
	\end{tikzcd}
	$$
	Then $(f_{\vert U})^*=(g_{\vert U})^*\colon H(Y) \to H(U)$, which implies that every element $\omega$ in $\mathcal{J}$ belongs to $\ker (i_U)^*=\im j_U$, then $\omega=j_U(\tilde\omega)$ for some $\tilde\omega\in H(X,U)$.
	
	Now, let us remember the relative cup product \cite[p. 209]{HATCHER}
	$$\smile\colon H^{m}(X ,U)  \otimes   H^{n}(X, V) \rightarrow H^{m+n}(X , U \cup V),$$
	where $U,V$ are open subsets of $X$. 
	From  \cite[p. 251]{SPANIER} it follows that the following diagram is commutative:
	%	\begin{equation*}
	%	%\label{}
	%	\begin{tikzcd}
	%	H(X,U)  \otimes   H(X , V) \arrow[r, "\smile"]  \arrow[d, "{j_U}\otimes {j_V}"] & H(X , U \cup V) \arrow[d, "{j_{U\cup V}}"] \\
	%	H(X )  \otimes   H(X ) \arrow[r,"\smile"] &   H(X ) 
	%	\end{tikzcd}
	%	\end{equation*}
	%
	%By induction, the following diagram is commutative too:
	
	$$%\begin{equation}\nonumber
	%\label{}
	\begin{tikzcd}
	H(X ,U_0)  \otimes \cdots  \otimes  H(X , U_n) \arrow[r, "\smile"]  \arrow[d, "{j_{U_0}}\otimes \cdots  \otimes {j_{U_n}}"] & H^{*}(X , \bigcup_{k=0}^n U_k ) \arrow[d, "{j_{U_0 \cup \ldots \cup U_n}}"] \\
	H(X)  \otimes \cdots  \otimes H(X ) \arrow[r,"\smile"] &   H(X) 
	\end{tikzcd}\\
	$$%\end{equation}
	
	If $ \omega_0 \smile \cdots \smile \omega_n$ 
	is a product of length $n+1$ of elements $\omega_k\in \mathcal{J}$, there exist elements $\widetilde{\omega_k}\in H(X,U_k)$ such that $\omega_k=j_k(\widetilde{\omega_k})$, hence
	$$ \omega_0 \smile \cdots \smile \omega_n=j_0(\widetilde{\omega_0})\smile\cdots\smile j_n(\widetilde{\omega_n})=j_{0...n}(\widetilde{\omega_0}\smile\cdots\smile \widetilde{\omega_n})=0$$
	because $\widetilde{\omega_0}\smile\cdots\smile \widetilde{\omega_n}\in H(X,X)=0$.
	Then $\lcp\mathcal{J}\leq n$.\end{proof}

\begin{example}
	Consider the inclusion maps $i_1,i_2\colon X\to X\times X$ as in Proposition \ref{INCLCAT}.  Then $\mathcal{J}(i_1,i_2)$ is isomorphic to $H(X)$. Therefore, we recover the classical cohomological lower bound for the LS category.
\end{example}

\begin{example} Let $G=U(2)$ be the Lie group of $2\times 2$ complex matrices $A$ such that $A ^{-1}=A^*$. It is known \cite{SINGHOF} that $\cat (G)=2$. In fact, topologically $G$ is the product $S^1\times S^3$, so its real cohomology is $H(G)=H(S^1)\otimes H(S^3)$, the exterior algebra $\bigwedge (x_1,x_3)$. Consider the maps $f=\id_G$ the identity and $g=I$ the inversion $I(A)=A^*$. Then $I^*(x_1)=-x_1$ and $I^*(x_3)=-x_3$, so $\mathcal{J}(f,g)=H(G)$. Then 
	\begin{equation}\label{U2}\nonumber
	2=\lcp \mathcal{J}\leq \D(f,g) \leq \cat(G)=2.
	\end{equation}
	Hence, the distance between the identity and the inversion is $2$, and the distance between the multiplication and the division is $2$ too (Proposition  \ref{MULTDIV}). 
	The same argument applies to the groups $U(n)$, $n\geq 2$, that is, $\D(\id_G,I)=n=\cat(G)$.
\end{example}
\subsection{Homotopy weight}
Following the ideas of Fadell-Husseini
%\cite{Fadell-Husseini}inencontrable
for the LS-category and Farber-Grant for the Topological Complexity \cite{FARBER-GRANT}, we can define a notion of {\em homotopy weight} that serves to improve inequality \eqref{LCP}. Our proofs follow the lines of those in \cite[Section 6]{FGSYM} for the $\TC$-weight, which is a particular case.

Let $f,g\colon X \to Y$ be two maps, and let $u\in H(X;R)$ be a cohomology class. 
\begin{definition}
	We say that $u$ has homotopy weight $\hw(u)= k+1$ (with respect to $f,g$) if $k$ is the greatest integer such that the following condition is satisfied: given any continuous map $\phi\colon A \to X$   with $D(f\circ \phi, g\circ \phi)\leq k$, then $\phi^*u=0\in H(A;R)$. We put $\hw(0)=\infty$. \end{definition}
In other words,  $\hw(u)\geq k+1$ means that $\phi^*u=0\in H(A;R)$ for all maps $\phi\colon A \to X$   with $D(f\circ \phi, g\circ \phi)\leq k$. 

We first prove the homotopy invariance of the homotopy weight.
\begin{proposition}If $\alpha,\beta$ in the following commutative diagram are homotopy equivalences, 
	$$
	\begin{tikzcd}
	X^\prime \arrow[r,shift left, "f^\prime"] \arrow[r, shift right,"g^\prime"']   \arrow[d, "\alpha"] & Y^\prime \arrow[d, "\beta"] \\
	X \arrow[r,shift left, "f"]\arrow[r,shift right,"g"'] &   Y
	\end{tikzcd}
	$$
	then $\hw_{f^\prime,g^\prime}(\alpha^*u)=\hw_{f,g}(u)$ for $u\in H(X)$.
\end{proposition}
\begin{proof}Let $\hw^\prime(\alpha^*u)\geq k+1$, and consider $\phi \colon A\to X$ such that $\D(f\phi,g\phi)\leq k$.
	If $\bar\alpha\colon X \to X^\prime$ is the homotopy inverse of $\alpha$, then  (Corollaries \ref{EQUIV1} and  \ref{EQUIV2}),
	$$\D(f^\prime\bar\alpha\phi,g^\prime\bar\alpha\phi)=\D(\beta f^\prime \bar\alpha\phi, \beta g^\prime \bar\alpha\phi)=D(f\alpha\bar\alpha\phi,g\alpha\bar\alpha\phi)=\D(f,g)\leq k$$
	because $\alpha\bar\alpha\phi\simeq \phi$, so 
	$0=(\bar\alpha\phi)^*(\alpha^*u)=\phi^*u$. This proves that $\hw(u)\geq k+1$. The other implication is analogous.
\end{proof}

Our invariant can be seen as a generalization of those introduced by several authors, including Rudyak \cite{RUDYAK2} and Strom \cite{STROM}. 

From now on we shall assume that our cohomology classes are in $\mathcal{J}(f,g)=\im (f^*-g^*)$ because it is there where we can assume that $\hw$ is well defined, as the following Lemma proves.

\begin{lemma}If $u\in \mathcal{J}(f,g)$ then $\hw(u)\geq 1$.\end{lemma}
\begin{proof}If $\D(f\phi,g\phi)=0$ then $f\phi\simeq g\phi$, so
	$\phi^*(f^*-g^*)=(f\phi)^*-(g\phi)^*=0$. Since $u=(f^*-g^*)v$ for some $v\in H(Y;R)$, we have $\phi^*u=0$, and the result follows.\end{proof}

\begin{lemma}For any non-zero class $u\in \mathcal{J}(f,g)\subset H(X)$ we have $\hw(u)\leq\D(f,g).$\end{lemma}
\begin{proof}If $\hw(u)\geq \D(f,g)+1=k+1$, then $\D(f\circ \id_X,g\circ \id_X)\leq k$, so $u=\id_X^*u=0$.\end{proof}

\begin{theorem}\label{WEIGHT}Let $u=u_0\smile\cdots\smile u_k$ be a cup product of cohomology classes in $\mathcal{J}(f,g)$. Then
	$$\hw(u)\geq \sum_{j=0}^k\hw(u_j).$$
\end{theorem}

\begin{proof} It is enough to prove the result when $k=1$. Let $\hw(u_0)=m+1$ and $\hw(u_1)=n+1$. We want to prove that $\hw(u_0\smile u_1)\geq m+n+2$. Let $\phi\colon A \to X$ such that $\D(f\phi,g\phi)\leq m+n+1$, then there exists an open covering $\{U_0,\dots,U_{m+n+1}\}$  of $A$ such that $f\phi_{\vert U_j}\simeq g\phi_{\vert U_j}$ for all $j$.  Define $V_0=U_0\cup\cdots\cup U_m$ and $V_1=U_{m+1}\cup\dots\cup U_{m+n+1}$. Then $\D(f\phi_{\vert V_0}, g\phi_{\vert V_0})=m$, so $\phi_{\vert V_0}^*u_0=0$. Analogously, $\D(f\phi_{\vert V_0}, g\phi_{\vert V_0})=n$ implies $\phi_{\vert V_1}^*u_1=0$.
	
	Now,  we consider the long exact sequence of the pair $(A,V_0)$, 
	$$\begin{tikzcd}
	& & {H}^{m}(X) \arrow{d}[swap]{\phi^*}\\
	\cdots \arrow[r] & {H}^{m}(A,V_0) \arrow[r,"{j_0}"] &  {H}^{m}(A)  \arrow[r,"i_0^*"] & {H}^{m}(V_0) \arrow[r] & \cdots
	\end{tikzcd}
	$$
	Since $i_0^*\phi^*u_0=0$, there exists $\xi_0\in H(A)$ such that $j_0(\xi)=\phi^*(u_0)$. Analogously,  $\phi^*(u_1)=j_1(\xi_1)$. Then, as in the proof of Theorem \ref{LCP}, we have  
	\begin{align*}
	\phi^*(u_0\smile u_1)&= \phi^*(u_0)\smile\phi^*(u_1)=j_0(\xi_0)\smile j_1(\xi_1)\\
	&=j_{01}(\xi_0\smile\xi_1)=0\in H(A;V_0\cup V_1).
	\end{align*}
\end{proof}

Theorem \ref{LCP} can be read as follows: if $u=u_1\smile\dots\smile u_k\neq 0$ is a non-zero product of $k$ cohomology classes in $\mathcal{J}(f,g)$, then $D(f,g)>k$.
Combining the latter Lemmas and Proposition we have proved:

\begin{theorem}\label{LOWB}If $u_0\smile\cdots\smile u_k\neq 0$ is a non-zero product of $k+1$ cohomology classes in $\mathcal{J}(f,g)$ then $\D(f,g)\geq\sum_{j=0}^k \hw(u_j)$.
\end{theorem}

The interest of this result is that it is possible to find 
elements of high category weight. For instance, we can mimic Theorem 6 in \cite{FARBER-GRANT}, originally stated only for the TC weight. For simplicity we only consider Steenrod squares, but it is possible to state it in a much larger context for other cohomology operations.

Let $\theta=\Sq^i\colon H^p(X;\Z_2) \to H^{p+i}(X;\Z_2)$, for $0\leq i \leq p$, be a Steenrod square.  Each one of these squaring operations is a morphism of abelian groups that  is natural and commutes with the connecting morphisms in the Mayer-Vietoris sequence. Its {\em excess} equals $i$, so $\theta(u)=0$ if $u\in H^{n-1}(X;\Z_2)$ with $n\leq i$ \cite[Theorem 1]{MOSHERTANGORA}.

\begin{theorem}\label{FHLCP}If $n\leq i$ and $u\in \mathcal{J}(f,g)\subset H^n(X;\Z_2)$, then $\hw(\theta(u))\geq 2$.\end{theorem}

\begin{proof}Let $\phi\colon A \to X$ with $D(f\phi,g\phi)\leq 1$, so $A=U_0\cup U_1$ with $f\phi_{\vert U_j}\simeq g\phi_{\vert U_j}$, for $j=0,1$. Consider the Mayer-Vietoris sequence
	$$\begin{tikzcd}[column sep=small]
	\dots\arrow[r] &H^{n-1}(U_0\cap U_1)\arrow[r,"\delta"]&H^n(A)\arrow[r]&H^n(U_0)\oplus H^n(U_1)\arrow[r]&\dots
	\end{tikzcd}$$
	
	Since $\phi^*u=\phi^*(f^*-g^*)v=\big(f\phi)^*-(g\phi)^*\big)v$ is zero on each $U_j$, there exists $w\in H^{n-1}(U_0\cap U_1)$ such that $\delta \omega =\phi^*u$. But then
	$$\phi^*(\theta u)=\theta(\phi^*u)=\theta(\delta \omega)=\delta(\theta \omega)=0,$$
	where the nullity holds because $\omega$ has degree $n-1<i$.
\end{proof}
\section{Fibrations}\label{FIBRATIONS}
\subsection{Statement of results}
A well known result from Varadarajan \cite{VARADARAJAN} states that if $\pi\colon E \to B$ is a (Hurewicz) fibration with generic fiber $F$ and path-connected base $B$, then
\begin{equation}\label{FORMULVAR}
\cat(E)+1\leq \big(\cat(B)+1\big)\big(\cat(F)+1\big).
\end{equation}

On the other hand, Farber and Grant \cite{FARBER-GRANT} proved that 
\begin{equation}\label{FGTC}
\TC(E)+1\leq \big(TC(F)+1\big)\times \cat \big(B\times B\big).
\end{equation}

We shall see that both results are particular cases of a much more general situation. Let $\pi^\prime\colon E^\prime \to B^\prime$ be another fibration with path-connected base $B^\prime$ and generic fibre $F^\prime$, and take two fiber-preserving maps $f,g\colon E \to E^\prime$, with induced maps $\bar f,\bar g\colon B \to B^\prime$. That is, we have $\pi^\prime \circ f=\bar f\circ\pi$ and $\pi^\prime \circ g=\bar g\circ\pi$, as in the  commutative diagram below:
$$
\begin{tikzcd}
E\arrow[d,"\pi"']\arrow[r, shift left,"f"]\arrow[r, shift right,"g"']&E^\prime\arrow[d,"\pi^\prime"]\\
B\arrow[r, shift left,"\bar f"]\arrow[r, shift right,"\bar g"']&B^\prime
\end{tikzcd}
$$

Our aim is to prove the following result:
\begin{theorem}\label{FIBRATION}Let $b_0\in B$ with $\bar f(b_0)=b_0^\prime=\bar g(b_0)$. If $f_0,g_0\colon F_0\to F_0^\prime$ are the induced maps between the fibers of $b_0$ and $b_0^\prime$,  then
	$$\D(f,g)+1\leq \big(\D(f_0,g_0)+1\big)\times \big(\cat(B)+1\big).$$\end{theorem}

It is well known that all the fibers $F_b=\pi^{-1}(b)$ of the fibration $\pi$ have the same homotopy type \cite[ Proposition  4.61]{HATCHER}. Also it is known that if the base $B$ is contractible then the fibration is fiber homotopy equivalent
to a product fibration \cite[Corollary~4.63]{HATCHER}. We need a similar statement that will allow us to establish our notations.

\begin{lemma}\label{TRIVCAT} If $U$ is a categorical open set in $B$, which contracts to the point $b_0$, then the fibration $\pi^{-1}(U)\to U$ is fiber homotopy equivalent to the trivial fibration $F_0\times U$. Moreover, on each fiber $F_b$, $b\in U$, the restriction of the homotopy equivalence is a homotopy equivalence $F_b\simeq F_0$.
\end{lemma}

\begin{proof}
	There is a homotopy $\mathcal{C}\colon U\times I\to B$ with $\mathcal{C}_0$ the inclusion $U\subset B$ and $\mathcal{C}_1$  the constant map $b_0\colon U \to B$. Then,
	the homotopy lifting property in the following diagram
	\begin{equation}\label{LIFT}
	\begin{tikzcd}
	\pi^{-1}(U)\times \{0\} \arrow[d,"i_0"']\arrow[hookrightarrow]{rr}& &E\arrow[d,"\pi"]\\
	\pi^{-1}(U)\times I \arrow[urr,shift right,dashrightarrow,"\widetilde{\mathcal{C}}"]\arrow[r,"\pi\times \id"']& U\times I\arrow[r,"\mathcal{C}"]& B
	\end{tikzcd}
	\end{equation}
	gives us a map 
	$
	\widetilde{\mathcal{C}}\colon \pi^{-1}(U)\times I \to E
	$
	such that $\widetilde{\mathcal{C}_0}$ is the inclusion $\pi^{-1}(U)\subset E$ and $\pi\widetilde{\mathcal{C}}(x,t)=\mathcal{C}(\pi(x),t)$.  As a consequence, we have a map (we use the same name with a slight abuse of notation) 
	\begin{equation}\label{CONTF0}\widetilde{\mathcal{C}_1}\colon \pi^{-1}(U)\to F_0=\pi^{-1}(b_0).
	\end{equation}
	For each $b\in U$, the path $\mathcal{C}_t(b)$ in $B$ connects the points $b$ and $b_0$, and it lifts, for each $x\in F_b$, to the path $\widetilde{\mathcal{C}_t}(x)$, so the map
	$
	(\widetilde{\mathcal{C}_1})_{\vert F_b}\colon F_b \to F_0
	$
	is the usual one giving the homotopy equivalence between the fibers. 
	
	% 
	% Conversely, consider the homotopy lifting property in the following diagram
	%$$\begin{tikzcd}
	%F_0\times U\times \{0\} \arrow[d,"i_0"']\arrow[r,"p_1"]  &E\arrow[d,"\pi"]\\
	%F_0\times U\times I \arrow[ur,shift right,dashrightarrow,"\widetilde{{\bar{\mathcal{C}}}}"]\arrow[r,"\bar{\mathcal{C}}"']&   B
	%\end{tikzcd}$$
	%where the maps $\bar{\mathcal{C}}$ and $p_1$ are defined as:
	%$$\bar{\mathcal{C}}(y,b,t)=\mathcal{C}(b,1-t),\quad p_1(y,b,0)=y.$$
	% For each $b\in U$, the path $\bar{\mathcal{C}}_t(b)=\mathcal{C}_{1-t}(b)$ connecting $b_0$ to $b$  lifts to the paths $\widetilde{{\bar{\mathcal{C}}}}_t(y,b)$, which means that the map
	%$\widetilde{{\bar{\mathcal{C}}}}_{1,b}\colon F_0 \to F_b$
	%is a homotopy inverse of the map $\widetilde{\mathcal{C}_{1,b}}$, by the usual construction. 
	
	The rest of the proof is similar to the usual one.
\end{proof}
\subsection{Proof of Theorem \ref{FIBRATION}} 
\begin{proof}We now start the proof of Theorem \ref{FIBRATION}. Assume that $\cat(B)\leq m$ and let $B=U_0\cup\cdots\cup U_m$ be a covering by categorical open sets. For each $U=U_i$ take the homotopy $\widetilde{\mathcal{C}}$ in \eqref{LIFT} and the map ${\widetilde{\mathcal{C}}}_1$ in \eqref{CONTF0}. If $\D(f_0,g_0)\leq n$, let $F_0=V_0\cup\cdots\cup V_n$ with $(f_0)_{\vert V_j}\simeq (g_0)_{\vert V_j}$. For each $V=V_j\subset F_0$ take the open set
	$$\Omega(U,V)=\pi^{-1}(U)\cap (\widetilde{\mathcal{C}_1})^{-1}(V)\subset E.$$
	We claim that $\Omega=\Omega(U,V)$ is is a homotopy domain for $f$ and $g$. To see it, we have that the inclusion   $\Omega \subset E$ is homotopical to the map $\widetilde{\mathcal{C}_1}_{\vert \Omega}$ by the homotopy $\widetilde{\mathcal{C}}_{\vert \Omega}$, hence $f_{\vert \Omega}\simeq (f\circ \widetilde{\mathcal{C}_1})_{\vert \Omega}$ and $g_{\vert \Omega}\simeq (g\circ \widetilde{\mathcal{C}_1})_{\vert \Omega}$. But $\widetilde{\mathcal{C}_1}(\Omega)$   is contained in $V\subset F^\prime_0$, and there is a homotopy $\mathcal{F}\colon V\times I\colon F_0\to F_0^\prime $ between $(f_0)_{\vert V}$ and $(g_0)_{\vert V}$, hence $i_0^\prime\circ \mathcal{F}$ is a homotopy between $(i_0^\prime\circ f_0)_{\vert V}$ and $(i_0^\prime\circ g_0)_{\vert V}$, where $i_0^\prime\colon F_0^\prime\subset E^\prime$ is the inclusion. We now define a map $\mathcal{G}\colon \Omega \times I \to E^\prime$ as
	$$\mathcal{G}(x,t)=
	\left\{\begin{array}{ll}
	(f\circ \widetilde{\mathcal{C}})(x,3t),&\quad \mathrm{if}\  0\leq t\leq 1/3,\\
	(i_0^\prime\circ \mathcal{F})(\widetilde{\mathcal{C}_1}(x),3t-1),&\quad \mathrm{if}\  1/3\leq t\leq 2/3,\\
	(g\circ \widetilde{\mathcal{C}})(x,3-3t),&\quad \mathrm{if}\  2/3\leq t\leq 1.
	\end{array}
	\right.$$
	This map is continuous because for $t=1/3$ we have
	$$(f\circ \widetilde{\mathcal{C}})_1(x)=(i_0^\prime\circ f_0)(\widetilde{\mathcal{C}_1}(x))=(i_0^\prime\circ \mathcal{F}_0)(\widetilde{\mathcal{C}_1}(x)),$$
	and for $t=2/3$ we have
	$$(i_0^\prime\circ \mathcal{F}_1)(\widetilde{\mathcal{C}_1}(x))=(i_0^\prime\circ g_0)(\widetilde{\mathcal{C}_1}(x))= (g\circ \widetilde{\mathcal{C}})_1(x).$$
	
	Finally,
	$\mathcal{G}(x,0)= (f\circ \widetilde{\mathcal{C}})_0(x)=f(x)$,
	while
	$\mathcal{G}(x,1)= (g\circ \widetilde{\mathcal{C}})_0(x)=g(x)$.
	%showing that $f_{\vert \Omega}\simeq g_{\vert \Omega}$, as we claimed.
	
	It is clear that $\{\Omega(U_i,V_j)\}$ is an open covering of $E$, so the result follows.
\end{proof}

\begin{example}By taking $E=E^\prime$, $B=B^\prime$, $f=\id_E$, $\bar f=\id_B$, $g=e_0$ a constant map and $\bar g$ the constant map $b_0=\pi(e_0)$ one obtains
	$$\D(\id_E,e_0)+1\leq \big(D(\id_0,e_0)+1)\times (\cat(B)+1),$$
	and we recover Formula \eqref{FORMULVAR}.
\end{example}

\begin{example}\label{MAINTC}In Theorem \ref{FIBRATION}, take the projections  $p_1,p_2$ as in the following diagram:
	$$\begin{tikzcd}
	E\times E\arrow[d,"\pi\times\pi"']\arrow[r, shift left,"p_1"]\arrow[r, shift right,"p_2"']&E\arrow[d,"\pi"]\\
	B\times B\arrow[r, shift left,"p_1"]\arrow[r, shift right,"p_2"']&B
	\end{tikzcd}$$
	and use  Theorem \ref{PROJECT}. We have $p_1(b_0,b_0)=b_0=p_2(b_0,b_0)$. Since  $(p_1)_0,(p_2)_0\colon F_0\times F_0 \to F_0$ are the projections, Formula \eqref{FGTC} follows.\end{example}

\section{Example}\label{MAINEXAMPLE}
In this section we show an example of a Lie group $G$ and  two maps $f,g\colon G \to G$ such that $\D(f,g)=2=\lcp \Ho(G)$, while $\TC(G)=\cat(G)=3$. 
\subsection{Description of the example}Let $G=\Sp(2)$ be the Lie group of $2\times 2$ quaternionic matrices such that $AA^*=I$ (where $A^*$ denotes the conjugate transpose and $I$ is the identity matrix). Its dimension is $10$. Its cohomology is $\Ho(G)=\Lambda(x^3,x^7)$, the exterior algebra with two generators, so the length of the cup product is $\lcp \Ho(G)=2$.
It is also known that $\TC(G)=\cat(G)=3$ \cite{SCHW}.

Let  $f=\mu_2\colon G \to G$ be  the map $\mu_2(A)=A^2$ and let $g=I$ be  the constant map $I$. We have $f^*\omega=2\omega$ for any bi-invariant form, while $g^*=0$. Then
$$\lcp \mathcal{J}(f,g)=2\leq \D(f,g)\leq 3=\cat(G).$$ 
We want to find a covering of $G$ by three open sets where these maps are homotopic, so
$$\D(\mu_2,I)=2.$$
\subsection{Auxiliary function}\label{AUXFUNCTION}
Our idea is to build  homotopies by means of the gradient flow of  an auxiliary function, namely the function $h\colon G \to \R$ given by the real part of the trace,
$$h(A)=\Re\Tr(A).$$
The reader can find more information about this map in \cite{FRANKEL}. It represents the height of $A$ over an hyperplane, when $G$ is viewed as a submanifold of $\R^{16}$ endowed with the Euclidean metric $\vert X\vert^2=\Re\Tr(X^*X)$.  It is a Morse-Bott function, whose critical set is
$$\Crit{h}=\{B\in G\colon B^2=I\}.$$
It is known that every matrix in $G$ is diagonalizable;  its (right) eigenvalues are quaternions of norm one; and  they are grouped in similarity classes, so they can be assumed to be complex numbers $z$, with $\vert z  \vert=1$ \cite{ZHANG}. 

Hence, the critical set $\Crit h$ is formed by the matrices of $G$ having real eigenvalues $\pm 1$.
That set has three connected components: the two points $\{\pm I\}$ and the  the orbit $\Sigma=\{UPU^*\colon U\in \Sp(2)\}$ of the matrix
\begin{equation}\label{MATRIZP}
P=
\left[\begin{array}{cc}
1&0\\
0&-1\\
\end{array}
\right]
\end{equation}
under the action of the group onto itself by conjugation. Hence $\Sigma\cong \Sp(2)/(\Sp(1)\times \Sp(1))$ is a compact Grassmannian manifold with $\dim\Sigma=4$. 
Note that $-P\in\Sigma$.

To end this preliminaries, we describe the foliated local structure of the gradient flow near the critical set  (see Figure \ref{FLOW1}): 
\begin{figure}[h]
	\begin{center}
		\begin{tikzcd}
			&\arrow[bend right=60]{dd}I\arrow[bend right=30]{dr}&\\
			&&\arrow[bend left=30]{dl}\Sigma\\ 
			&-I&\\   
		\end{tikzcd}
		\caption{The flow of the negative gradient of $h$}
		\label{FLOW1}
	\end{center}
\end{figure}
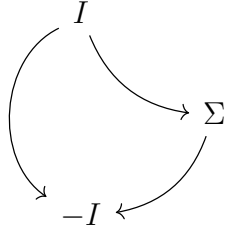

The stable manifold $W^+(\Sigma)$ of the critical submanifold $\Sigma$ (that is, the points of $G$ whose flow line ends at $\Sigma$) fibers over $\Sigma$, and this fiber bundle
$W^+(\Sigma) \to \Sigma$ is isomorphic to the positive normal bundle $p^+\colon \nu^+(\Sigma) \to \Sigma$ defined as follows: for the critical point $B=UPU^*\in\Sigma$, the Hessian $H_B\colon T_BG \to T_BG$ is given by  \cite{GOMMACPER}:
$$H_B(X)=-\frac{1}{2}(XB+BX).$$ 
Its kernel is the tangent space $T_B\Sigma$ to the critical orbit, and the normal space $\nu_B\Sigma$ decomposes as $\nu^+_B\oplus \nu^-_B$ depending on the sign of the eigenvalues of $H_B$. All these constructions are invariant by conjugation, so a simple computation for the particular case $B=P$ shows that 
$\nu_B^+= U\nu^+_PU^*$,  where $X\in \nu^+_P$ if and only if 
$X=
\left[\begin{array}{cc}
0&0\\
0&b
\end{array}
\right]$, with $\Re(b)=0$, so $\dim \nu^+_B=3$.
%\draw [thick] (0,0) ellipse [x radius = 1cm, y radius = 0.5cm];
Hence, the index of the critical manifold $\Sigma$ equals $3$.
Analogously, $W^s(-I)$ is a cell of dimension $10$. 
\subsection{Contractible open sets}\label{CAYLEYCONTRACTIBLE}
We also need to describe the explicit categorical covering of $\Sp(2)$ given by the first author  in \cite{GOMMACPER}. For that, for each  of the matrices $B=\pm I,\pm P$ (see \eqref{MATRIZP}), let us consider the open neighbourhood $$\Omega_G(B)=\{A\in G\colon B+A \mathrm{\ invertible}\}.$$

There are global diffeomorphisms (called Cayley transforms)
$$c_B\colon \Omega_G(B) \to T_{B^*}G,$$
given by
$$c_B(A)=(I-B^*A)(B+A)^{-1},$$
which prove that each $\Omega_G(B)$ is a contractible open set. The inverse map of $c_B$ is $c_{B^*}$.
Moreover
\begin{equation}\label{DOSABS}
\Omega(I)\cup\Omega(-I)= G\setminus \Sigma
\end{equation}
and 
$$\Sigma \subset \Omega(P)\cup\Omega(-P).$$
\subsection{Homotopies} Let 
$U_0=G\setminus (\Sigma\cup \{I\})$.
Since $h$ is a Morse-Bott function, there is a well defined map  sending each point $A\in U_0$ to the final point $\phi_A(+\infty)\in \Sigma\cup \{-I\}$ of the flow line $\phi_A(t)$  passing through it. This map is not continuous, but we shall use the fact that  $\mu_2(\Crit h)=\{I\}$.

Let us describe the following three open sets, covering $G$, where the maps $f=\mu_2$ and $g=I$ will be homotopic:
\begin{enumerate}
	\item
	The first one is $U_0=G\setminus (\Sigma\cup \{I\})$.
	We define 
	$$\mathcal{H}\colon U_0\times [0,+\infty] \to G$$ to be the map 
	$$\mathcal{H}(A,t)=
	\left\{
	\renewcommand{\arraystretch}{1.2}
	\begin{array}{ll}
	\phi_A(t)&\quad\mathrm{if\ }\  0\leq t < +\infty,\\
	\lim\limits_{t\to \infty}\phi_A(t)&\quad\mathrm{if\ } \  t= +\infty.\\
	\end{array} \right.
	$$
	
	Then $\mu_2\circ\mathcal{H}_0=(\mu_2)_{\vert U_0}$ while $\mu_2\circ\mathcal{H}_\infty$ is the constant map $I$. The explicit formulas for the flow $\phi_A(t)$ given in \cite{GOMMACPER} and the description of the stable set $W^s$ that we gave in \ref{AUXFUNCTION} allow to prove that the map $\mu_2\circ\mathcal{H}$ is continuous.
	
	%It remains to show that the map $\mu_2\circ\mathcal{H}$ is continuous (see Lemma \ref{CRUCIAL} below).
	\item
	The second open set is the Cayley domain $U_1=\Omega_G(I)$ (see \ref{CAYLEYCONTRACTIBLE}), which is contractible, hence $\mu_2\simeq I$ on it.
	
	\item
	Finally, the third one will be a tubular open neighbourhood $U_2=N(\Sigma)$ of the critical Grassmannian manifold $\Sigma$, in such a way that $\Sigma$ is a deformation retract of $N(\Sigma)$, by a retraction $$\mathcal{R}\colon N(\Sigma)\times [0,1] \to N(\Sigma)\subset G.$$ That is,  $\mathcal{R}_0$ is the inclusion $N(\Sigma)\subset G$; and the image of $\mathcal{R}_1$ is contained in $\Sigma$. Since the square of any critical point is $I$, we have that $\mu_2\circ \mathcal{R}\colon N(\Sigma) \times [0,1] \to G$ is a homotopy between $\mu_2\circ\mathcal{R}_0=(\mu_2)_{\vert N(\Sigma)}$ and the constant map $\mu_2\circ \mathcal{R}_1=I$.
\end{enumerate}
It is clear that $G=U_0\cup U_1\cup U_2$, so $\D(I,\mu_2)\leq 2$, as stated.

\section{Further ideas}\label{FURTHER}
\subsection{Contiguity distance between simplicial maps} It is easy to adapt our definitions to the simplicial setting. For instance, in \cite{QUIQUETC,QUIQUECAT2,QUIQUECAT1} simplicial versions of LS-category and topological complexity were given by one of the authors. With the classical notion of contiguous simplicial maps replacing that of homotopical continuous maps, one can define a notion of distance between simplicial  maps. 

\begin{definition}
	The {\em contiguity distance} $\SD(\varphi,\psi)$  between two simplicial maps $\varphi,\psi\colon K \to K^\prime$  is the least integer $n\geq 0$ such that there is a covering of $K$ by subcomplexes $K_0,\dots,K_n$ such that the  restrictions $\varphi_{\vert K_j},\psi_{\vert K_j}\colon K_j \to K^\prime$ are in the same contiguity class, for all $j=0,\dots,n$.  If there is no such covering, we define $\SD(f,g)=\infty$.
\end{definition} 

As expected, this notion of contiguity distance generalizes those of simplicial LS category $\scat(K)$ and discrete topological complexity $\TC(K)$: 
\begin{example}
	Given two simplicial complexes $K$ and $L$, denote by $K \prod L$ their categorical product \cite{KOZLOV}. The contiguity distance between the projections $p_1,p_2\colon K \prod K\to K$ equals $\TC(K)$, as follows from \cite[Theorem 3.4]{QUIQUETC}.
\end{example}

\subsection{Higher homotopic distance} The notion of topological complexity has been extended to higher analogs \cite{RUDYAK}. The same can be done for the homotopy distance.

\begin{definition}
	Given $m$ continuous maps $f_1,\ldots,f_m \colon X \to Y$, their {\em $m$-th homotopy distance} $\D(f_1,\ldots,f_m)$ is the least integer $n\geq 0$ such that there exists a covering of $X$  by open subspaces $\{U_0,\dots,U_n\}$, such that  the restrictions $f_{1_{\vert U_j}}\simeq\ldots \simeq f_{m_{\vert U_j}}\colon U_j \to Y$, for all $j=0,\dots,n$.
\end{definition} 

We denote the  $m$-th topological complexity of the space $X$ by $\TC_m(X)$. As expected, the notion of $m$th homotopic distance generalizes the notion of higher topological complexity: 

\begin{theorem}
	Given a path-connected topological space $X$, consider the projections $$p_1,\ldots,p_m\colon X\times \stackrel{m)}\cdots \times X \to X.$$ Then $\D(p_1,\ldots,p_m)=\TC_m(X)$.
\end{theorem}

\bibliographystyle{plain}
\bibliography{biblio_v10}

\begin{thebibliography}{10}

\bibitem{BARMAK}
Jonathan~A. {Barmak}.
\newblock {\em {Algebraic topology of finite topological spaces and
  applications.}}, volume 2032.
\newblock Berlin: Springer, 2011.

\bibitem{CORNEA}
Octav Cornea, Gregory Lupton, John Oprea, and Daniel Tanr\'{e}.
\newblock {\em Lusternik-{S}chnirelmann category}, volume 103 of {\em
  Mathematical Surveys and Monographs}.
\newblock American Mathematical Society, Providence, RI, 2003.

\bibitem{FARBER}
Michael Farber.
\newblock Topological complexity of motion planning.
\newblock {\em Discrete Comput. Geom.}, 29(2):211--221, 2003.

\bibitem{FARBER2}
Michael Farber.
\newblock Instabilities of robot motion.
\newblock {\em Topology Appl.}, 140(2-3):245--266, 2004.

\bibitem{FGSYM}
Michael Farber and Mark Grant.
\newblock Symmetric motion planning.
\newblock In {\em Topology and robotics}, volume 438 of {\em Contemp. Math.},
  pages 85--104. Amer. Math. Soc., Providence, RI, 2007.

\bibitem{FARBER-GRANT}
Michael Farber and Mark Grant.
\newblock Robot motion planning, weights of cohomology classes, and cohomology
  operations.
\newblock {\em Proc. Amer. Math. Soc.}, 136(9):3339--3349, 2008.

\bibitem{QUIQUETC}
D.~Fern\'{a}ndez-Ternero, E.~Mac\'{i}as-Virg\'{o}s, E.~Minuz, and J.~A.
  Vilches.
\newblock Discrete topological complexity.
\newblock {\em Proc. Amer. Math. Soc.}, 146(10):4535--4548, 2018.

\bibitem{QUIQUECAT2}
D.~Fern\'andez-Ternero, E.~Mac\'{\i}as-Virg\'os, E.~Minuz, and J.A. Vilches.
\newblock Simplicial {L}usternik-{S}chnirelmann category.
\newblock {\em Publicacions Matematiques}, 63:265--293, 2019.

\bibitem{QUIQUECAT1}
D.~Fern\'{a}ndez-Ternero, E.~Mac\'{i}as-Virg\'{o}s, and J.~A. Vilches.
\newblock Lusternik-{S}chnirelmann category of simplicial complexes and finite
  spaces.
\newblock {\em Topology Appl.}, 194:37--50, 2015.

\bibitem{FRANKEL}
Theodore Frankel.
\newblock Critical submanifolds of the classical groups and {S}tiefel
  manifolds.
\newblock In {\em Differential and {C}ombinatorial {T}opology ({A} {S}ymposium
  in {H}onor of {M}arston {M}orse)}, pages 37--53. Princeton Univ. Press,
  Princeton, N.J., 1965.

\bibitem{GOMMACPER}
A.~G\'{o}mez-Tato, E.~Mac\'{i}as-Virg\'{o}s, and M.~J. Pereira-S\'{a}ez.
\newblock Trace map, {C}ayley transform and {LS} category of {L}ie groups.
\newblock {\em Ann. Global Anal. Geom.}, 39(3):325--335, 2011.

\bibitem{HATCHER}
Allen Hatcher.
\newblock {\em Algebraic topology}.
\newblock Cambridge University Press, Cambridge, 2002.

\bibitem{JAMES}
I.~M. {James}.
\newblock {On \(H\)-spaces and their homotopy groups.}
\newblock {\em {Q. J. Math., Oxf. II. Ser.}}, 11:161--179, 1960.

\bibitem{KOZLOV}
Dmitry Kozlov.
\newblock {\em Combinatorial algebraic topology}, volume~21 of {\em Algorithms
  and Computation in Mathematics}.
\newblock Springer, Berlin, 2008.

\bibitem{LS}
Gregory Lupton and J\'{e}r\^{o}me Scherer.
\newblock Topological complexity of {$H$}-spaces.
\newblock {\em Proc. Amer. Math. Soc.}, 141(5):1827--1838, 2013.

\bibitem{MOSHERTANGORA}
Robert~E. Mosher and Martin~C. Tangora.
\newblock {\em Cohomology operations and applications in homotopy theory}.
\newblock Harper \& Row, Publishers, New York-London, 1968.

\bibitem{JOHNJEFF}
John Oprea and Jeff Strom.
\newblock Mixing categories.
\newblock {\em Proc. Amer. Math. Soc.}, 139(9):3383--3392, 2011.

\bibitem{RUDYAK2}
Yuli~B. Rudyak.
\newblock On category weight and its applications.
\newblock {\em Topology}, 38(1):37--55, 1999.

\bibitem{RUDYAK}
Yuli~B. Rudyak.
\newblock On higher analogs of topological complexity.
\newblock {\em Topology Appl.}, 157(5):916--920, 2010.

\bibitem{SCHW}
Paul~A. Schweitzer.
\newblock Secondary cohomology operations induced by the diagonal mapping.
\newblock {\em Topology}, 3:337--355, 1965.

\bibitem{SINGHOF}
Wilhelm Singhof.
\newblock On the {L}usternik-{S}chnirelmann category of {L}ie groups.
\newblock {\em Math. Z.}, 145(2):111--116, 1975.

\bibitem{SPANIER}
Edwin~H. Spanier.
\newblock {\em Algebraic topology}.
\newblock McGraw-Hill Book Co., New York-Toronto, Ont.-London, 1966.

\bibitem{STROM}
Jeffrey~Andrew Strom.
\newblock {\em Category weight and essential category weight}.
\newblock ProQuest LLC, Ann Arbor, MI, 1997.
\newblock Thesis (Ph.D.)--The University of Wisconsin - Madison.

\bibitem{TANAKA}
Kohei {Tanaka}.
\newblock {A combinatorial description of topological complexity for finite
  spaces.}
\newblock {\em {Algebr. Geom. Topol.}}, 18(2):779--796, 2018.

\bibitem{VARADARAJAN}
K.~Varadarajan.
\newblock On fibrations and category.
\newblock {\em Math. Z.}, 88:267--273, 1965.

\bibitem{ZHANG}
Fuzhen Zhang.
\newblock Ger\v{s}gorin type theorems for quaternionic matrices.
\newblock {\em Linear Algebra Appl.}, 424(1):139--153, 2007.

\end{thebibliography}

\end{document}